\documentclass[runningheads]{llncs}

\usepackage[T1]{fontenc}
\usepackage{amsmath,amssymb,mathtools}
\usepackage{graphicx}
\usepackage{tikz}
\usetikzlibrary{positioning,arrows.meta,calc}
\usepackage[inline]{enumitem}
\usepackage[colorlinks=true,citecolor=blue,linkcolor=blue,urlcolor=blue]{hyperref}
\usepackage{aliascnt}
\usepackage{comment}

\newaliascnt{observation}{theorem}

\aliascntresetthe{observation}

\newcommand{\crn}{\operatorname{cr}}
\newcommand{\eps}{\varepsilon}

\begin{document}

\title{Tight Upper Bounds on Color Reversal by Local Inversions\thanks{Supported by ANRF MATRICS Grant MTR/2022/000692: ``Algorithmic study on hereditary graph properties''. A preliminary version of this paper has been accepted to IWOCA 2026.}}

\author{Hitendra Kumar\inst{1} \and
Kumud Singh Porte\inst{2} \and
R. B. Sandeep\inst{2}}

\authorrunning{Kumar et al.}

\institute{Indian Institute of Science Education and Research, Pune, India\\
\email{hitendra.kumar1729@gmail.com}
\and
Indian Institute of Technology Dharwad, India\\
\email{\{cs24dp012,sandeeprb\}@iitdh.ac.in}}

\maketitle

\begin{abstract}
A bicoloration of a graph $G=(V,E)$ is a map
$\beta:V\to\{-1,1\}$. A local inversion at a vertex $v$
complements the subgraph induced by the neighbors of $v$ and
simultaneously reverses the colors of all neighbors of $v$.
Sabidussi (Discrete Mathematics, 1987) showed that every bicolored
graph on $n$ vertices without isolated vertices admits a color reversal
using at most $6n+3$ local inversions, and that any two bicolorings of
such a graph can be transformed into each other using at most $9n$
local inversions. Recently, Porte, Sandeep, and Santra (CALDAM 2026)
improved these bounds to $4n-3$ and
$\lfloor(11n-3)/2\rfloor$, respectively. We prove the tight bound
$3n$ by showing that, for every graph on $n$ vertices without isolated
vertices, any bicoloring can be transformed into any other bicoloring
using at most $3n$ local inversions. We also show that this bound is
best possible: for complete graphs and stars on $n$ vertices, at least
$3n$ local inversions are required to reverse the colors of all
vertices. Moreover, the proof of the upper bound is constructive: given two bicolorings, it produces, in polynomial time, a sequence of at most $3n$ local inversions transforming one into the other.
\keywords{local inversion \and color reversal \and bicolored graph \and local complementation \and graph operations}
\end{abstract}

\section{Introduction}

Local graph operations that modify adjacency relations, vertex labels, or both are fundamental in structural and algorithmic graph theory. A central example is \emph{local complementation}: at a vertex $v$, one complements the subgraph induced by $N_G(v)$. Local complementation has been studied extensively because of its connections with circle graphs, rank-width, vertex-minors, isotropic systems, and quantum information theory; see, for example, \cite{12,13,11,10,14,7,16,8,9,15,DBLP:journals/jct/Oum05,seidel1991survey}. When vertices carry signs or colors, it is natural to ask how such local changes interact with the labels. This leads to the operation studied in this paper, called \emph{local inversion}, which combines local complementation with color reversal on the neighborhood of the selected vertex.

We consider simple undirected labeled graphs whose vertices are assigned colors from $\{-1,1\}$. A \emph{bicolored graph} is a pair $B=(G,\beta)$, where $G$ is a graph and $\beta:V(G)\to\{-1,1\}$ is a bicoloring. A local inversion at $v\in V(G)$ complements the graph induced by $N_G(v)$ and changes the sign of every vertex in $N_G(v)$. Applying a sequence of local inversions produces a new bicolored graph on the same vertex set, but generally with a different underlying graph.

The \emph{color-reversal number} $\crn(G)$ is defined as the minimum length of a string $w$ of vertices such that, for every bicoloring $\beta$ of $G$, applying the local inversions in $w$ returns the underlying graph to $G$ and changes every color. Since the effect of a fixed string on the color vector is independent of the initial signs, the existence of such a string for one bicoloring is equivalent to its existence for all bicolorings.

The problem was introduced by Sabidussi~\cite{sabidussi1987color}, who proved that if $G$ has no isolated vertices, then $\crn(G)\le 6n+3$. Sabidussi also studied the more general transformation problem between two bicolorings of the same graph and proved an upper bound of $9n$ local inversions. Porte, Sandeep, and Santra~\cite{porte2026} improved these bounds to $4n-3$ and $\lfloor(11n-3)/2\rfloor$, respectively, using parity-aware decompositions based on perfect forests and decompositions of odd trees. They also observed that complete graphs and stars admit color reversal in $3n$ moves and asked whether this is optimal, and whether the $3n$ bound holds for all graphs.

We answer these questions affirmatively. We prove that, given a graph $G$ without isolated vertices and two bicolorings of $G$, there is a sequence of at most $3n$ local inversions transforming one bicoloring into the other. Furthermore, we prove that this bound is tight for complete graphs and stars with at least three vertices. Our proof of the upper bound is constructive, and the sequence is produced in polynomial time. We decompose a spanning tree into vertex-disjoint stars and then reverse colors locally on each part. The exact lower bounds for complete graphs and stars use a structural description of the possible graph states reached from a clique or a star under local inversions.

\section{Preliminaries}\label{sec:prelim}

All graphs in this paper are finite, simple, undirected, and labelled. For a graph $G$, we write $V(G)$ and $E(G)$ for its vertex and edge sets. For $X\subseteq V(G)$, $G[X]$ denotes the subgraph induced by $X$. The open neighborhood of a vertex $v$ in $G$ is denoted by $N_G(v)$. We write $P_t$ and $S_t$ for the path and the star on $t$ vertices, respectively.

Let $G=(V,E)$ and let $a\in V$. The \emph{local complement} of $G$ at $a$, denoted $G*a$, is the graph obtained from $G$ by complementing the subgraph induced by $N_G(a)$ and leaving all other adjacencies unchanged. Thus, for distinct vertices $x,y\in V$,
\[
xy\in E(G*a)
\quad\Longleftrightarrow\quad
\begin{cases}
xy\notin E(G), & \text{if } x,y\in N_G(a),\\
xy\in E(G), & \text{otherwise.}
\end{cases}
\]

A \emph{string} over $V(G)$ is a finite sequence of vertices of $G$. The empty string is denoted by $\eps$. If $w=x_1x_2\cdots x_m$, then $|w|=m$. For a graph $G$ and a string $w$, the graph $G*w$ is defined recursively by $G*\eps=G$ and $G*(w a)=(G*w)*a$.

A \emph{bicoloration} of $G$ is a map $\beta:V(G)\to\{-1,1\}$, and a \emph{bicolored graph} is a pair $B=(G,\beta)$. If $B=(G,\beta)$ and $a\in V(G)$, the \emph{local inversion} of $B$ at $a$ is
\[
B_a=(G*a,\beta_a),
\]
where
\[
\beta_a(x)=
\begin{cases}
-\beta(x), & \text{if } x\in N_G(a),\\
\beta(x), & \text{otherwise.}
\end{cases}
\]
Here the neighborhood is taken in the current underlying graph. For a string $w=x_1\cdots x_m$, we write $B_w$ for the bicolored graph obtained by applying the local inversions at $x_1,\ldots,x_m$ in this order.

For $A\subseteq V(G)$, let $B^A$ denote the bicolored graph obtained from $B$ by changing exactly the colors of the vertices in $A$ and leaving the underlying graph unchanged. Thus $B^A=(G,h_A(\beta))$, where
\[
(h_A(\beta))(v)=
\begin{cases}
-\beta(v), & v\in A,\\
\beta(v), & v\notin A.
\end{cases}
\]
For a singleton $\{a\}$ we write $B^a$ instead of $B^{\{a\}}$.

Two strings $w,w'\in V(G)^*$ are said to be \emph{equivalent over $B$}, written $w\sim_B w'$, if $B_w=B_{w'}$. When the bicolored graph is clear from the context, we simply write $w\sim w'$. We also write $w\simeq_G w'$ if $G*w=G*w'$. In particular, if $w\simeq_G\eps$, then $B_w=B^A$ for some $A\subseteq V(G)$.

The following facts will be used repeatedly in our proofs.

\begin{proposition}[Sabidussi~\cite{sabidussi1987color}]\label{prop:involution}
For every bicolored graph $B=(G,\beta)$ and every vertex $a\in V(G)$, $B_{aa}=B$. Equivalently, $aa\sim\eps$.
\end{proposition}

\begin{proposition}[Sabidussi~\cite{sabidussi1987color}]\label{prop:edge}
Let $B=(G,\beta)$ and let $ab\in E(G)$. Then
\[
B_{ababab}=B^{\{a,b\}}.
\]
\end{proposition}

\begin{proposition}[Sabidussi~\cite{sabidussi1987color}]\label{prop:triangle-one}
Let $B=(G,\beta)$, and suppose that $a,b,c$ form a triangle in $G$. Then
\[
B_{abacbac}=B^a.
\]
\end{proposition}

\begin{proposition}[Sabidussi~\cite{sabidussi1987color}]\label{prop:triangle-three}
Let $B=(G,\beta)$, and suppose that $a,b,c$ form a triangle in $G$. Then
\[
B_{ababcbcac}=B^{\{a,b,c\}}.
\]
\end{proposition}



\section{The upper bound}\label{sec:upper}

We first obtain an upper bound on the length of string to reverse colors on all vertices of a graph.
Then we use it to prove an upper bound for transforming one bicoloring to another.
We start with two lemmas dealing with reversing the colors of induced stars. The first is the basic three-vertex case.

\begin{lemma}\label{lem:p3-three}
Let $B=(G,\beta)$, and suppose that $b,c\in N_G(a)$ and $bc\notin E(G)$. Then
\[
B_{babcbcaca}=B^{\{a,b,c\}}.
\]
\end{lemma}

\begin{proof}
Set $u=ababcbcac$. Since $b,c\in N_G(a)$ and $bc\notin E(G)$, the vertices $a,b,c$ form a triangle in $G*a$. By \autoref{prop:triangle-three}, applying $u$ to $B_a$ reverses exactly the colors of $a,b,c$ relative to $B_a$ and preserves the graph $G*a$. Therefore $B_{a u a}=B^{\{a,b,c\}}$. The string $a u a$ begins with $aa$, and after cancelling this pair by \autoref{prop:involution} we obtain the string $babcbcaca$.\qed
\end{proof}

\begin{lemma}\label{lem:star-upper}
Let $B=(G,\beta)$, and let $X\subseteq V(G)$ induce a star on $s\ge 2$ vertices with center $c$. Then there is a string $w$ over $X$ of length at most $3s$ such that $B_w=B^X$.
\end{lemma}

\begin{proof}
Let the leaves of the star be $c_1,\ldots,c_{s-1}$.

If $s=2$, the assertion follows from \autoref{prop:edge}. Hence assume $s\ge 3$.

Suppose first that $s$ is odd. Apply \autoref{lem:p3-three} to the induced path $c_1cc_2$. This gives a graph-preserving string of length $9$ that reverses the colors of $c,c_1,c_2$. Next apply a local inversion at $c$. This changes the induced star on $X$ into a clique on $X$ and flips all leaves. Pair the remaining leaves as $c_3c_4,c_5c_6,\ldots,c_{s-2}c_{s-1}$ and, for each pair, use \autoref{prop:edge} in the clique. Finally apply a local inversion at $c$ again. The two inversions at $c$ restore the underlying graph, and the net color change on $X$ is exactly the reversal of all vertices in $X$. The resulting string has length
\[
9+1+3(s-3)+1=3s+2.
\]
The last symbol of the first length-$9$ string is $c$, and it is immediately followed by another $c$; cancelling this consecutive pair by \autoref{prop:involution} yields an equivalent string of length $3s$.

Now suppose that $s$ is even. Use \autoref{prop:edge} on the edge $cc_1$, then apply a local inversion at $c$, pair the remaining leaves as $c_2c_3,c_4c_5,\ldots,c_{s-2}c_{s-1}$ and reverse every pair using \autoref{prop:edge}, and finally apply a local inversion at $c$. The same color-counting argument shows that all vertices of $X$ are reversed and the underlying graph is restored. The length before cancellation is
\[
6+1+3(s-2)+1=3s+2.
\]
Again, the final $c$ in the edge-reversal string is followed by another $c$, and cancelling this pair gives an equivalent string of length $3s$.\qed
\end{proof}

The next lemma is a standard bottom-up decomposition of a rooted tree into vertex-disjoint stars.

\begin{lemma}\label{lem:star-decomposition}
Every tree $T$ with at least two vertices can be decomposed in polynomial time into vertex-disjoint stars, each with at least two vertices, whose vertex sets partition $V(T)$.
\end{lemma}

\begin{proof}
Root $T$ at an arbitrary vertex $r$. We construct the stars greedily from the leaves upward. As long as vertices remain, choose a remaining vertex $v$ of maximum depth that has at least one remaining child. Let $C$ be the set of remaining children of $v$, and initially put $X=\{v\}\cup C$. Remove $X$ from the remaining tree. If the parent $p$ of $v$ exists and becomes isolated after this removal, then add $p$ to $X$ as well and remove it. The set $X$ induces a star, centered at $v$: its leaves are the vertices of $C$, and possibly also $p$. The process removes at least two vertices in each step and never disconnects the unprocessed part except by deleting completed leaf blocks. Hence it terminates with a partition of $V(T)$ into stars of size at least two. The construction is plainly polynomial.\qed
\end{proof}

We now prove the main upper bound in a slightly stronger localized form. This form will also be used in the transformation theorem.

\begin{theorem}\label{thm:localized-upper}
Let $B=(G,\beta)$ be a bicolored graph, and let $H$ be a connected induced subgraph of $G$ with $m\ge 2$ vertices. Then there exists a string $w$ over $V(H)$ of length at most $3m$ such that
\[
B_w=B^{V(H)}.
\]
\end{theorem}

\begin{proof}
Let $T$ be a spanning tree of $H$. By \autoref{lem:star-decomposition}, $T$ has a vertex partition into stars $T_1,\ldots,T_q$, each of size at least two. We show that, for each $i$, the colors of exactly the vertices of $T_i$ can be reversed using at most $3|V(T_i)|$ local inversions while preserving the whole graph $G$.

Fix one star $T_i$ with center $c$ and vertex set $X$. If $|X|=2$, the claim follows from \autoref{prop:edge}. Hence assume $|X|\ge 3$. Let $L=X\setminus\{c\}$ and let $M$ be a maximum matching in $G[L]$.

If $M$ is a perfect matching of $G[L]$, choose an edge $uv\in M$. For each edge in $M\setminus\{uv\}$, use \autoref{prop:edge}; this reverses all leaves except $u$ and $v$. Since $u$ and $v$ are adjacent and both adjacent to $c$, the vertices $c,u,v$ form a triangle, and \autoref{prop:triangle-three} reverses exactly these three vertices. The total length is
\[
6(|M|-1)+9=6\left(\frac{|X|-1}{2}-1\right)+9=3|X|.
\]

Suppose now that $M$ is not perfect. Let $U=L\setminus V(M)$. Since $M$ is maximum, $U$ is independent in $G[L]$. Therefore $G[\{c\}\cup U]$ is an induced star centered at $c$. Reverse every matched pair using \autoref{prop:edge}, and then reverse the induced star $G[\{c\}\cup U]$ using \autoref{lem:star-upper}. The total length is
\[
6|M|+3(1+|U|)=3(1+2|M|+|U|)=3|X|.
\]

Thus each part $T_i$ can be processed independently within the claimed budget. Concatenating the resulting graph-preserving strings gives a string reversing exactly $V(H)$ of length at most
\[
\sum_{i=1}^q 3|V(T_i)|=3m.
\]\qed
\end{proof}

We obtain \autoref{cor:cr-upper} by Applying \autoref{thm:localized-upper} to each connected component of $G$ and concatenating the resulting strings.

\begin{corollary}\label{cor:cr-upper}
If $G$ is a graph on $n$ vertices with no isolated vertices, then $\crn(G)\le 3n$.
\end{corollary}

Before proving a matching upper bound for transforming one bicoloring into another, we obtain the following simple lemmas, one to flip the color of the center
of an induced $P_3$ and the other to flip the color of a degree-1 vertex.

\begin{lemma}\label{lem:p3-center}
Let $B=(G,\beta)$, and suppose that $b,c\in N_G(a)$ and $bc\notin E(G)$. Then there is a string of length $7$ over $\{a,b,c\}$ such that $B_w=B^a$.
\end{lemma}

\begin{proof}
Set $u=abacbac$. As in the proof of \autoref{lem:p3-three}, the vertices $a,b,c$ form a triangle in $G*a$. By \autoref{prop:triangle-one}, applying $u$ to $B_a$ reverses exactly the color of $a$ relative to $B_a$ and preserves $G*a$. Thus $B_{a u a}=B^a$. Cancelling the initial $aa$ gives the length-$7$ string $bacbaca$.\qed
\end{proof}

\begin{lemma}\label{lem:pendant}
Let $B=(G,\beta)$, and let $b$ be a pendant vertex with unique neighbor $a$. Then
\[
B_{ababa}=B^b.
\]
\end{lemma}

\begin{proof}
Let $N=N_G(a)\setminus\{b\}$. Track the effect of the string $ababa$. The first inversion at $a$ flips the colors of $b$ and all vertices in $N$, and makes $b$ adjacent to every vertex of $N$. The following inversion at $b$ flips the colors of $a$ and all vertices in $N$ and toggles the adjacencies between $a$ and $N$ and inside $N$. The middle inversion at $a$ now sees only $b$ as a neighbor and flips only $b$. The fourth inversion, again at $b$, reverses the changes made by the second inversion, and the final inversion at $a$ reverses the changes made by the first inversion. Consequently the underlying graph is restored, the vertices in $N\cup\{a\}$ are flipped an even number of times, and $b$ is flipped an odd number of times. Hence $B_{ababa}=B^b$.\qed
\end{proof}

\begin{theorem}\label{thm:transformation-connected}
Let $G$ be a connected graph on $n\ge 2$ vertices, and let $B=(G,\beta)$ and $B'=(G,\beta')$ be two bicolored graphs with the same underlying graph $G$. Then there exists a string $w$ of length at most $3n$ such that $B_w=B'$.
\end{theorem}

\begin{proof}
Let
\[
V_0=\{v\in V(G):\beta(v)=\beta'(v)\},\qquad
V_1=\{v\in V(G):\beta(v)=-\beta'(v)\}.
\]
It suffices to reverse exactly the vertices of $V_1$.

If $V_1=\emptyset$, there is nothing to prove. If $V_0=\emptyset$, the claim follows from \autoref{cor:cr-upper}. Hence assume that both sets are nonempty.

Let $I$ be the set of isolated vertices of $G[V_1]$, and let $C=V_1\setminus I$. Each connected component of $G[C]$ has at least two vertices. By \autoref{thm:localized-upper}, each such component $Q$ can be reversed using at most $3|Q|$ local inversions, without affecting any vertex outside $Q$. After processing all components of $G[C]$, it remains to reverse exactly the vertices in $I$.

We now process $I$ iteratively. Let $A$ be the set of vertices of $V_0$ that are still available for the length accounting; initially $A=V_0$. The following three rules are applied until $I$ becomes empty.

\smallskip
\noindent\textbf{Rule 1.} Suppose some $x\in A$ has at least two neighbors in $I$. Let $R=N_G(x)\cap I$. Since the vertices of $I$ are isolated in $G[V_1]$, the set $R$ is independent. Apply a local inversion at $x$, reverse the clique induced by $R$ using \autoref{thm:localized-upper}, and then apply a local inversion at $x$ again. This reverses exactly the colors of the vertices in $R$ and restores the graph. Its length is at most $2+3|R|\le 3(|R|+1)$. Remove $R$ from $I$ and remove $x$ from $A$.

\smallskip
\noindent\textbf{Rule 2.} Suppose Rule 1 does not apply and some $y\in I$ has at least two neighbors in $A$. Choose distinct $x_1,x_2\in N_G(y)\cap A$. If $x_1x_2\in E(G)$, then $x_1,x_2,y$ form a triangle and \autoref{prop:triangle-one} gives a length-$7$ string reversing only $y$. If $x_1x_2\notin E(G)$, then $x_1yx_2$ is an induced $P_3$ with center $y$, and \autoref{lem:p3-center} gives a length-$7$ string reversing only $y$. In either case the length is at most $7\le 3\cdot 3$. Remove $y$ from $I$ and remove $x_1,x_2$ from $A$.

\smallskip
\noindent\textbf{Rule 3.} Suppose neither Rule 1 nor Rule 2 applies. We claim that every remaining vertex $z\in I$ has exactly one neighbor in $A$. Since $G$ is connected and $z$ has no neighbor in $V_1$, it has at least one neighbor in $V_0$. Moreover, a vertex removed from $A$ in Rule 1 had all its neighbors in the then-current set $I$ processed at the same time, and a vertex removed from $A$ in Rule 2 had no other neighbor in the then-current set $I$ because Rule 1 did not apply. Thus no removed vertex of $V_0$ is adjacent to a still-unprocessed vertex of $I$. Hence $z$ has a neighbor in $A$. Since Rule 2 does not apply, it has exactly one such neighbor, say $x$. Therefore $z$ is pendant in $G$, with unique neighbor $x$. By \autoref{lem:pendant}, the string $xz x z x$ has length $5\le 3\cdot 2$ and reverses only $z$. Remove $z$ from $I$ and remove $x$ from $A$.

Each rule reverses at least one remaining vertex of $I$, preserves the underlying graph, and does not disturb any vertex already processed. Thus the procedure terminates and reverses exactly $V_1$.

It remains to bound the length. The first phase contributes at most $3|C|$. In the second phase, the charged sets are respectively $R\cup\{x\}$ in Rule 1, $\{y,x_1,x_2\}$ in Rule 2, and $\{z,x\}$ in Rule 3. These charged sets are pairwise disjoint, and each rule uses at most three times the size of its charged set. Therefore the second phase contributes at most $3(|I|+|V_0|)$. Altogether,
\[
|w|\le 3|C|+3|I|+3|V_0|=3n.
\]
This completes the proof.\qed
\end{proof}

By applying \autoref{thm:transformation-connected} independently to each connected component and concatenating the resulting strings, we obtain
\autoref{cor:transformation-no-isolates}.

\begin{corollary}\label{cor:transformation-no-isolates}
Let $G$ be a graph on $n$ vertices with no isolated vertices. Any two bicolorings of $G$ can be transformed into each other using at most $3n$ local inversions.
\end{corollary}

\section{Lower bounds}\label{sec:complete}

We next show that the upper bound is best possible for complete graphs and star graphs.

\begin{lemma}\label{lem:first-return-clique}
Let $n\ge 3$, and let $B=(G,\beta)$ be a bicolored complete graph on
$n$ vertices. Let $\beta'\ne \beta$ be another bicoloring of $G$, and let
$w$ be a shortest string of local inversions such that applying $w$ to
$B$ results in $B'=(G,\beta')$. Let $w''$ be the shortest nonempty prefix
of $w$ such that, after applying $w''$ to $B$, the resulting bicolored
graph has underlying graph $G$. Then
\[
        w''=aba
\]
for some two distinct vertices $a,b\in V(G)$. Moreover, if
$(G,\beta'')$ is the bicolored graph obtained from $B$ by applying
$w''$, then
\[
        \beta''(v)=-\beta(v) \quad \text{if and only if} \quad v=a .
\]
\end{lemma}

\begin{proof}
Since $\beta'\ne\beta$, the string $w$ is nonempty. Let the first symbol
of $w$ be $a$. Since the underlying graph of $B$ is the complete graph
$G=K_n$, the neighborhood of $a$ in $G$ is $V(G)\setminus\{a\}$. Thus a
local inversion at $a$ complements the complete graph induced by
$V(G)\setminus\{a\}$ and leaves all edges incident with $a$ unchanged.
Hence, after applying the first operation, the underlying graph is the
star with center $a$ and leaf set $V(G)\setminus\{a\}$.

We now consider the operations before the underlying graph first becomes
complete again. While the underlying graph is this star, a local inversion
at a leaf $b\ne a$ does not change the underlying graph, because the
neighborhood of $b$ consists only of the single vertex $a$. Such an
operation only changes the color of $a$. In contrast, a local inversion
at the center $a$ complements the independent set induced by the leaves,
and therefore restores the complete graph.

Since $w''$ is the shortest nonempty prefix of $w$ after which the
underlying graph is again $G$, it follows that $w''$ must have the form
\[
        w'' = a b_1 b_2 \cdots b_m a,
\]
where $m\ge 0$ and each $b_i$ is a vertex distinct from $a$.

We claim that $m=1$. If $m=0$, then $w''=aa$. Since local inversion is an
involution, the prefix $aa$ has no effect on either the underlying graph
or the coloring. Deleting this prefix from $w$ would therefore give a
shorter string with the same final result, contradicting the minimality
of $w$.

Now suppose that $m\ge 2$. During the intermediate operations
$b_1,b_2,\ldots,b_m$, the underlying graph remains the same star centered
at $a$. Each such operation changes only the color of $a$ and leaves the
underlying graph unchanged. Hence any two intermediate leaf inversions
cancel each other in their net color effect and have no effect on the
underlying graph. Deleting two of them therefore produces a shorter string
with exactly the same final bicolored graph, again contradicting the
minimality of $w$. Thus $m=1$.

Therefore
\[
        w''=aba
\]
for some vertex $b\ne a$.

It is straightforward to verify that the string $aba$ changes only the color of $a$, leaving both the underlying graph and the colors of all other vertices unchanged.
This proves the lemma.\qed
\end{proof}

\begin{theorem}\label{thm:complete-exact}
For every $n\ge 3$, we have
\[
        \crn(K_n)=3n .
\]
\end{theorem}

\begin{proof}
The upper bound $\crn(K_n)\le 3n$ follows from
\autoref{cor:cr-upper}. It remains to prove the lower bound.

Let $G=K_n$, and let $\beta$ be an arbitrary bicoloring of $G$. Let
\[
        B=(G,\beta).
\]
Let $w$ be a shortest string of local inversions which transforms $B$
into $(G,-\beta)$. We prove that $|w|\ge 3n$.

We decompose $w$ into consecutive first-return blocks to the clique.
More precisely, let $w_1$ be the shortest nonempty prefix of $w$ such
that, after applying $w_1$ to $B$, the underlying graph is again $G$.
By \autoref{lem:first-return-clique}, we have
\[
        w_1=a_1b_1a_1
\]
for two distinct vertices $a_1,b_1\in V(G)$, and the net effect of
$w_1$ is to change only the color of $a_1$.

After deleting this prefix, the remaining suffix of $w$ is still a
shortest string transforming the current bicolored clique into the final
bicolored clique $(G,-\beta)$; otherwise, replacing it by a shorter such
string would contradict the minimality of $w$. Hence, as long as the
current coloring is not yet $-\beta$, we may apply
\autoref{lem:first-return-clique} again to the remaining suffix. In this
way, $w$ is decomposed into consecutive blocks
\[
        w=w_1w_2\cdots w_t,
\]
where each block $w_i$ has the form
\[
        w_i=a_i b_i a_i
\]
with $a_i\ne b_i$, has length $3$, restores the underlying graph to
$G$, and changes the color of exactly one vertex, namely $a_i$.

Since the final coloring is $-\beta$, the color of every vertex of $G$
must be changed an odd number of times over the whole string $w$.
However, each block $w_i$ changes the color of exactly one vertex. Thus
the number of blocks must be at least $n$; indeed, each of the $n$
vertices must occur as $a_i$ in an odd, and hence nonzero, number of
blocks. Therefore
\[
        |w|=\sum_{i=1}^{t}|w_i|=3t\ge 3n .
\]
This completes the proof.\qed
\end{proof}

We now prove that the same lower bound holds for stars.

\begin{lemma}\label{lem:star-transition}
Let $V$ be a set of $n\ge 3$ vertices, and for each $x\in V$, let $S_x$
denote the star on $V$ with center $x$. Let $c,d\in V$ with $c\ne d$.
Assume that the current underlying graph is $S_c$. Suppose that a
substring of local inversions starts with the symbol $c$ and that, after
this first inversion, the next underlying graph in the sequence which is
a star is $S_d$. Then this substring is exactly
\[
        cd .
\]
Moreover, for every bicoloring $\beta$ of $S_c$, applying the string
$cd$ to $(S_c,\beta)$ produces a bicolored graph with underlying graph
$S_d$, and its net effect on the coloring is to reverse precisely the
colors of $c$ and $d$.
\end{lemma} 

\begin{proof}
After the first local inversion at $c$, the independent set induced by
the leaves of $S_c$ is complemented. Hence the underlying graph becomes
the complete graph on $V$.

From the complete graph, a local inversion at a vertex $x$ complements
the complete graph induced by $V\setminus\{x\}$. Therefore the resulting
underlying graph is the star $S_x$. Since, by assumption, the next star
obtained is $S_d$, the next operation must be the local inversion at
$d$. Thus the corresponding substring is exactly $cd$.

It remains to compute the effect on the coloring. The local inversion at
$c$ in the star $S_c$ changes the colors of all vertices in
$V\setminus\{c\}$. The subsequent local inversion at $d$ in the complete
graph changes the colors of all vertices in $V\setminus\{d\}$. Hence
each vertex in
\[
        V\setminus\{c,d\}
\]
is changed twice, while $c$ and $d$ are each changed exactly once.
Therefore the net effect of the string $cd$ is to reverse precisely the
colors of $c$ and $d$.
\end{proof} \qed

\begin{theorem}\label{thm:star-exact}
For every $n\ge 3$, we have
\[
        \crn(S_n)=3n .
\]
\end{theorem}

\begin{proof}
The upper bound $\crn(S_n)\le 3n$ follows from
\autoref{lem:star-upper}. We prove the lower bound.

Let $S_n$ be a star on vertex set $V$, and let its center be
$c_0$. Let $\beta$ be an arbitrary bicoloring of $S_n$, and let $w$ be a
shortest string of local inversions which transforms $(S_n,\beta)$ into
$(S_n,-\beta)$.

During the application of $w$, whenever the underlying graph is a star,
say $S_c$, an inversion at a leaf of $S_c$ leaves the underlying graph
unchanged and changes only the color of the center $c$. We call such an
operation a \emph{leaf inversion}. An inversion at the center $c$ changes
the underlying graph from the star $S_c$ to the complete graph on $V$.
The next time the underlying graph becomes a star, it is obtained by
applying a local inversion to the complete graph.

We call the corresponding substring a \emph{center transition}. More
precisely, if the process is at the star $S_c$, then applies a local
inversion at $c$, and the next star reached is $S_d$, we call this a
center transition from $c$ to $d$. Since $w$ is shortest, we may assume
that $d\ne c$: if $d=c$, then the transition is the string $cc$, which
has no effect on either the underlying graph or the coloring, and hence
could be deleted.

By \autoref{lem:star-transition}, every center transition from $c$ to
$d$, with $c\ne d$, consists exactly of the two local inversions $cd$.
Moreover, its net effect on the coloring is to reverse precisely the
colors of $c$ and $d$.

Let
\[
        c_0,c_1,\ldots,c_t
\]
be the sequence of centers of the successive star states reached through
center transitions. Since the initial and final underlying graphs are the
same star $S_n=S_{c_0}$, we have
\[
        c_t=c_0 .
\]
Thus the center transitions define a closed walk
\[
        c_0,c_1,\ldots,c_t=c_0
\]
in the complete graph on vertex set $V$.

We now examine the parity of the color changes. Fix a vertex $v\in V$.
By \autoref{lem:star-transition}, a center transition changes the color
of $v$ exactly when the corresponding edge of the closed walk is incident
with $v$. Since the walk is closed, the number of such incidences is even.
Therefore the total number of color changes contributed to $v$ by all
center transitions is even.

On the other hand, the final coloring is $-\beta$, so the total number of
color changes at every vertex must be odd. Hence, for each vertex
$v\in V$, there must be an odd, and in particular nonzero, number of leaf
inversions performed while the current star is centered at $v$. Since a
leaf inversion in the star $S_v$ changes only the color of $v$, this is
the only way to change the parity of the color changes at $v$ after the
even contribution coming from the center transitions.

Consequently, the process must visit a star centered at every vertex of
$V$. Hence the closed walk
\[
        c_0,c_1,\ldots,c_t=c_0
\]
visits all $n$ vertices. Since it is a closed walk with no loops and it
visits all $n$ vertices, it has length at least $n$. Therefore the number
of center transitions is at least $n$.

Each center transition contributes exactly two local inversions. In
addition, as shown above, there must be at least one leaf inversion for
each of the $n$ vertices. Hence
\[
        |w| \ge 2n+n=3n .
\]
This completes the proof.\qed
\end{proof}
\section{Concluding remarks}\label{sec:conclusion}

We proved that every graph on $n$ vertices without isolated vertices admits a global color reversal using at most $3n$ local inversions. The bound is tight, as witnessed by complete graphs and stars. We also proved the same $3n$ upper bound for transforming one bicoloring into another on a graph without isolated vertices.

Several natural questions remain open. First, for which graph classes (other than complete graphs and star graphs) is the bound $3n$ attained?
Second, if $\beta$ and $\beta'$ are two bicolorings of the same graph $G$, is the minimum number of local inversions needed to transform $\beta$ into $\beta'$ always bounded above by $\crn(G)$?
Algorithmic questions concerning these parameters also remain unexplored.
\bibliographystyle{splncs04}
\bibliography{CR_ref}

\end{document}